\documentclass [11pt]{article}
\usepackage{amsmath}
\usepackage{amsfonts}
\usepackage{amssymb}
\usepackage{amsthm}
 \usepackage{verbatim}
 \usepackage{euscript}
 \usepackage{inputenc}
 \usepackage[active]{srcltx}
\newtheorem{theorem}{Theorem}
\newtheorem{lemma}{Lemma}[section]
\newtheorem{proposition}{Proposition}[section]
\newtheorem{corollary}{Corollary}[section]

\newtheorem{acknowledgment*}{Acknowledgment}
\newtheorem{remark}{Remark}[section]

\numberwithin{equation}{section}
\newcommand{\rlemma}[1]{Lemma~\ref{#1}}
\newcommand{\rth}[1]{Theorem~\ref{#1}}
\newcommand{\rcor}[1]{Corollary~\ref{#1}}

\newcommand{\rprop}[1]{Proposition~\ref{#1}}

\newcommand{\e}{\varepsilon}
\newcommand{\R}{\mathbb R}
\newcommand{\C}{\mathbb C}

\newcommand{\OO}{\mathcal O}

\begin{document}
\bibliographystyle{siam}
\title{Radially symmetric minimizers for a $p$-Ginzburg Landau type energy in $\R^2$} \date{} \author{Yaniv
  Almog\thanks{Department of Mathematics, Louisiana State University,
    Baton Rouge, LA 70803, USA}, Leonid Berlyand\thanks{Department of
    Mathematics, Pennsylvania State University, University Park, PA
    16802, USA}, Dmitry Golovaty\thanks{Department of Theoretical and
    Applied Mathematics, The University of Akron, Akron, Ohio 44325,
    USA} and Itai Shafrir\thanks{Department of Mathematics, Technion -
    Israel Institute of Technology, 32000 Haifa, Israel}}
\maketitle
\begin{abstract}
  We consider the minimization of a p-Ginzburg-Landau energy functional
  over the  class of radially symmetric functions of degree one. We
  prove the existence 
  of a unique  minimizer in this class, and show that its modulus
  is monotone increasing and concave. We also study the
  asymptotic limit  of the minimizers as $p\to\infty$. Finally, we prove that
  the radially symmetric solution is
  locally stable for $2<p\leq4$.
\end{abstract}
\section{Introduction}
\label{sec:1}
Given $p>2$ consider the minimization problem of the energy functional
 \begin{equation}
\label{eq:38}
E_p(u)=\int_{\R^2}|\nabla u|^p+\frac{1}{2}(1-|u|^2)^2
\end{equation}
 over the class of maps $u\in W^{1,p}_{loc}(\R^2,\R^2)$ that satisfy
$E_p(u)<\infty$ and have a degree $d$ ``at infinity''. In our previous work
\cite{aletal09}
it was shown that the notion of degree at infinity is well-defined. Hence,
minimization over the homotopy class of maps with degree $d$ is a
sensible task.
Moreover, in the case of degree $d=1$ we proved that a minimizer does
exist. An important 
open question is whether  any minimizer $u$ is necessarily radially
symmetric, i.e., $u=f(r)e^{i\theta}$ for some function $f(r)$ satisfying
$f(0)=0$ (thanks to invariance with respect to translations we may
assume that $u(0)=0$). We show in the sequel that a (unique) minimizer
{\em within the radially symmetric class} $u_p=f_p(r)e^{i\theta}$ exists. We
were, however, unable to determine whether $u_p$ is a minimizer or
not. As a preliminary step towards establishing the minimality
properties of $u_p$, we study in the present paper its {\em stability} properties. 
One of
our main results (see Theorem~\ref{th:stable} below) establishes that
$u_p$ is indeed stable if $p\in(2,4]$. We conjecture that this result
remains valid for any $p>2$. It should be mentioned that the
analogous stability problem for $p=2$ on the disc $B_1(0)$  with the
boundary condition 
$u(z)=\frac{z}{|z|}$  on $\partial B_1(0)$ was solved by Mironescu
\cite{mi95} and in a weaker form, by Lieb and Loss \cite{liebloss94}. Going back to the
problem on $\R^2$, but again for $p=2$, we recall that  the $L^2$-
stability of the radially symmetric solution was proved by Ovchinnikov
and Sigal \cite{ovsi97} and in a more natural energy space by
 del Pino, Felmer and Kowalczyk \cite{deetal04}.
However,  
Mironescu \cite{mi96}  showed a stronger result, namely, that the
radially symmetric solution is the unique (up to rotations and
translations) {\em local
  minimizer} on $\R^2$, that is, on every disc $B_R(0)$ it is
minimizing for its boundary values on $\partial B_R(0)$. Note that for $p=2$
(in contrast with $p>2$) only the
notion of local minimizer makes sense since the admissible
maps have infinite energy.

The manuscript is organized as
follows. In Section~\ref{sec:3} we establish existence and uniqueness
of the minimizer $u_p=f_p(r)e^{i\theta}$ in the radially symmetric class, as well as its
regularity. We also show that  $f_p$ is increasing and concave and
obtain some precise estimates for $f_p(r)$ for large values of $r$. In
Section~\ref{sec:large} we study the limit of $f_p$ as $p$ tends to
infinity. We show that $\lim_{p\to\infty} f_p=f_\infty$ is the piecewise linear
function given by $\frac{r}{\sqrt{2}}$ for $r<\sqrt{2}$ and is
identically equal to $1$ for $r\geq\sqrt{2}$.
 Finally, Section~\ref{sec:4} is devoted to the study of the stability
 of the radially symmetric solution.
\section{Radially symmetric solutions}
 \label{sec:3}

 In this section we consider some of the properties of the
 minimizer of
 \begin{equation}
\label{eq:29}
   I_p(f) = \int_0^\infty \left\{ \Big[ (f^\prime)^2+\frac{f^2}{r^2}\Big]^{p/2} +
     \frac{1}{2}(1-f^2)^2 \right\} rdr
 \end{equation}
for any $p>2$. Note that $I_p(f)=\frac{1}{2\pi}E(u)$ where $u=f(r)e^{i\theta}$.

\subsection{Existence}
\label{sec:3.1}
For each $p>2$ we define the space
\begin{equation}
  \label{eq:37}
  X_p=\Big\{f\in W^{1,p}_{\text{loc}}(0,\infty)\,:\, \int_0^\infty \big(
  |f'|^2+\frac{f^2}{r^2}\big)^{p/2}\,rdr<\infty\Big\}\,.  
\end{equation}
Existence of a solution will be established by minimization of
$I_p(f)$ over $X_p$. Note that $X_p\subset C^\alpha_{\text{loc}}[0,\infty)$, with
$\alpha=1-2/p$,  since whenever $f\in X_p$,  the function
$F(x_1,x_2)=f(\sqrt{x_1^2+x_2^2})$ belongs to
$W^{1,p}_{\text{loc}}(\R^2)$, and then we can apply  Morrey's theorem.
Furthermore, for every $f\in X_p$ we must have have $f(0)=0$. This
follows from the continuity of $f$ and the fact that 
$$
\int_0^1 \frac{|f|^p}{r^{p-1}}<\infty\,.
$$ 
\begin{proposition}
  \label{prop:existence}
The minimum of $I_p(f)$ over $X_p$ is attained by a function $f_p\in
X_p$ satisfying $0\leq f_p(r)\leq 1,\,\forall r\in[0,\infty)$. 
\end{proposition}
\begin{proof}
  Put 
  \begin{equation}
\label{eq:94}
m_p=\inf_{f\in X_p} I_p(f)\,.
  \end{equation}
 We first note that $m_p<\infty$ since the function $g^*\in X_p$ defined
 by
 \begin{equation*}
   g^*(r)=
   \begin{cases}
     r & r\leq 1\,,\\
     1 & r>1\,,
   \end{cases}
 \end{equation*}
 verifies $I_p(g^*)<\infty$. Consider a minimizing sequence $\{g_m\}$ for
 \eqref{eq:38}, i.e.,
 \begin{equation*}
   \lim_{m\to\infty} I_p(g_m)=m_p\,.
 \end{equation*}
 By passing to a diagonal sequence we may assume that for any compact
 interval $[a,b]\subset(0,\infty)$ we have
 \begin{equation}
\label{eq:39}
   g_m\rightharpoonup g\text{ weakly in } W^{1,p}(a,b)\,.
 \end{equation}
Since the convexity of the Lagrangian 
\begin{equation*}
  L(P,Z,r)=\left\{ \Big(P^2+\frac{Z^2}{r^2}\Big)^{p/2} +
     \frac{1}{2}(1-Z^2)^2 \right\}r
\end{equation*}
 in the variable $P$ implies weak lower-semi-continuity of the
 functional  
 $I_p^{(a,b)}(f):=\int_a^bL(f',f,r)dr$ (see \cite[Theorem~1,Sec.~8.2]{ev98}),  we deduce from
 \eqref{eq:39} that
 \begin{equation}
\label{eq:40}
  I_p^{(a,b)}(g)\leq m_p\,.
 \end{equation}
 Since the interval $[a,b]$ is arbitrary, we conclude from
 \eqref{eq:40} that $g\in X_p$, $I_p(g)\leq m_p$, so that necessarily
 $I_p(g)=m_p$, and $g$ is a minimizer in \eqref{eq:38}. Since
 replacing $g$ by
 \begin{equation*}
   \tilde g(r)=\min(1,|g(r)|)\,,
 \end{equation*}
 gives a map $\tilde g\in X_p$ such that $I_p(\tilde g)\leq I_p(g)$ (with
 strict inequality, unless $|g|\leq 1$) we conclude that we
 may assume $0\leq g(r)\leq 1$ for all $r$, and the result follows for
 $f_p=g$.   
\end{proof}
The next lemma shows that $f$ is positive on $(0,\infty)$.
\begin{lemma}
\label{le1a}
  $f_p>0$ for all $r>0$.
\end{lemma}
\begin{proof}
We first claim that there is no interval of the form $[0,a]$, with
$a>0$ such that 
\begin{equation}
  \label{eq:64}
f\equiv 0 \text{ on }[0,a]\,.
\end{equation}
Indeed, suppose that \eqref{eq:64} holds for some $a$.   Fix any
function $g\in C^\infty[0,a]$ satisfying $g(0)=g(a)=0$ and $g(r)>0$ for $r\in
(0,a)$.  Then, for any
small $\e>0$ consider the function $h_\e$ defined by
\begin{equation*}
  h_\e(r)=\begin{cases}
  \e g(r) & 0\leq r\leq a\,,\\
   f_p(r)   & r>a\,.
\end{cases}
\end{equation*}
A simple computation gives
\begin{equation*}
  I_p(h_\e)=I_p^{(a,\infty)}(f_p)+\epsilon^p\int_0^a \big(
  |g'|^2+(\frac{g}{r})^2\big)^{p/2}\,rdr+\int_0^a (1-\epsilon^2g^2)^2\,rdr<I_p(f_p)\,,
\end{equation*}
 provided $\epsilon$ is chosen small enough. 
 
Next, we turn to the proof itself and  assume by negation that $f_p(r_0)=0$ for some
  $r_0>0$.  Put 
\begin{equation*}
\delta_0=\max_{r\in[0,r_0]} f_p(r)\,.    
\end{equation*}
 By the above claim $\delta_0>0$. Let $\delta\in(0,\delta_0)$ and  consider the set $S_\delta=\{
  r>0\,:\,f_p(r)<\delta\}$.  Denote by $J=(\alpha,\beta)$ the component of $S_\delta$
  containing $r_0$.  Since $\delta<\delta_0$ we have $\alpha>0$.   There is a $\delta_1>0$
  such that the function 
$$
H_r(t)=\Big(\frac{t}{r}\Big)^p +\frac{1}{2}(1-t^2)^2 
$$
is decreasing on $[0,\delta_1]$ for every $r\geq \alpha$ .
We may now replace $\delta$ by $\min(\delta,\delta_1)$ and set
\begin{displaymath}
  \tilde{f}(r) =
  \begin{cases}
    f_p & r\not \in J \\
    \delta & r\in J 
  \end{cases}\,.
\end{displaymath}
 From the monotonicity of $H_r$ it follows that $I_p(\tilde{f})< I_p(f_p)$. A contradiction.
\end{proof}
\subsection{Uniqueness}
\label{sec:3.2}
\begin{proposition}
 The non-negative minimizer for $I_p(f)$ is unique.
\end{proposition}
\begin{proof}
  We use a convexity argument due to Benguria (see \cite{bbl81}) for the case of the
  Laplacian (see \cite{bbl81})  and by Diaz and Sa\'a~\cite{disa87} and
  Anane~\cite{anane87} for
  the case of the $p$-Laplacian. More
  specifically,  we follow the presentation of Belloni and Kawhol~\cite{beka02}. 
 Assume $f$ and $g$ are both minimizers in \eqref{eq:29}. By an
 argument from the proof
 of \rprop{prop:existence}  it follows  that
 necessarily $f(r)\leq 1$ and $g(r)\leq1$ for each $r$.  Set
 \begin{equation*}
   \eta=\frac{f^p+g^p}{2}~\text{ and }~w=\eta^{\frac{1}{p}}\,.
 \end{equation*}
Denote also
\begin{equation*}
  s(r)=\frac{f^p}{f^p+g^p}\,.
\end{equation*}
Note that
\begin{equation*}
  w'=\frac{1}{2}\eta^{\frac1p-1}\big(f^{p-1}f^\prime+g^{p-1}g^\prime\big)\,.
\end{equation*}
Next we compute
\begin{multline*}
  ({w^\prime}^2+\frac{w^2}{r^2})^{\frac{p}{2}}=\Big|\eta^{\frac2p-2}\big(\frac{f^{p-1}f^\prime+g^{p-1}g^\prime}{2}\big)^2+\frac{\eta^{\frac2p}}{r^2}\Big|^{\frac p2}=
 \eta\Big|\big(\frac{f^{p-1}f^\prime+g^{p-1}g^\prime}{2\eta}\big)^2+\frac1{r^2}\Big|^{\frac
   p2}
\\=\eta\Big|\big(\frac{s(r)f^\prime}{f}+\frac{(1-s(r))g^\prime}{g}\big)^2+\frac1{r^2}\Big|^{\frac
  p2}
\\\leq \eta\Big(s(r){\big(\frac{|f'|^2}{f^2}+\frac1{r^2}\big)}^{\frac p2}+
(1-s(r)){\big(\frac{|g'|^2}{g^2}+\frac1{r^2}\big)}^{\frac p2}\Big)\\
=\frac{\eta}{f^p+g^p}\big(f'^2+\frac{f^2}{r^2}\big)^{\frac
  p2}+\frac{\eta}{f^p+g^p}\big(g'^2+\frac{g^2}{r^2}\big)^{\frac p2}
=\frac{1}{2}\Big(\big(f'^2+\frac{f^2}{r^2}\big)^{\frac
  p2}+\big(g'^2+\frac{g^2}{r^2}\big)^{\frac p2}\Big)
\end{multline*}
 Above we used the convexity of the function
 $t\mapsto(t^2+\frac{1}{r^2})^{\frac p2}$. Note that equality holds in the
 above only if $\frac{f'}{f}=\frac{g'}{g}$. If such an equality holds
 for all $r$, we conclude easily that $g=cf$ for some constant $c$,
 which then must be equal to $1$. 
 Therefore, the uniqueness claim follows from
 the above inequality and the convexity of the second term $(1-f^2)^2$
 as a function of $f^p$ for $p\geq2$ and $0\leq f\leq1$.
\end{proof}
\begin{remark}
  As a matter of fact, the only minimizers of $I_p$ are $f_p$ and
  $-f_p$. In view of lemma \ref{le1a} a non-negative minimizer must be
  strictly positive.  Since $I_p( |f| )=I_p(f)$,
  it follows that a minimizer may not change sign, and our assertion
  follows from the uniqueness for non-negative minimizers.
\end{remark}

\subsection{Regularity}
\label{sec:3.3}
This subsection is devoted to the study of the regularity properties
of the minimizer $f_p$.
\begin{proposition}
  \label{prop:regularity}
  We have $f_p\in C^\infty(0,\infty)$.
\end{proposition}
\begin{proof}
 The Euler-Lagrange equation associated with \eqref{eq:29} is 
   \begin{equation}
 \label{eq:30}
     \frac{1}{r}\big(r|\nabla u_p|^{p-2}f_p^{\prime}\big)^\prime =
     |\nabla u_p|^{p-2}\frac{f_p}{r^2} - \frac{2}{p}f_p(1-f_p^2) \, ,
   \end{equation}
where 
$$
u_p=f_p(r)e^{i\theta}\,.
$$
 A direct consequence of \eqref{eq:30} is that $|\nabla u_p|^{p-2}f_p^{'}\in
 W^{1,\frac{p}{p-2}}_{\text{loc}}(0,\infty)\subset C(0,\infty)$ and we
 immediately obtain that   $f_p\in C^1(0,\infty)$ (using that $f_p>0$ by \rlemma{le1a}) . Inserting this
 new information into \eqref{eq:30} we deuce that  $f_p\in
 C^2(0,\infty)$. Bootstrapping gives  $f_p\in C^k(0,\infty)$ for all $k$, as claimed.
\end{proof}

Our next objective is to prove the differentiability of $f$ at $0$.
\begin{proposition}
  \label{prop:at-zero}
 $f_p^{'}(0)=\lim_{r\to 0^+}\frac{f_p(r)}{r}$ exists and is a positive number.  
\end{proposition}
\begin{proof}
 We denote for convenience $f$ for $f_p$ and get from \eqref{eq:30},
 \begin{multline}
\label{eq:91}
   0=f''\Big( 1+\frac{p-2}{|\nabla
     u_p|^2}|f'|^2\Big)+\frac{f'}{r}\Big(1-\frac{p-2}{|\nabla
     u_p|^2}\frac{f^2}{r^2}\Big)-\frac{f}{r^2}\Big(1-\frac{p-2}{|\nabla
     u_p|^2}{|f'|}^2\Big)\\+\frac{2}{p}|\nabla u_p|^{2-p}f(1-f^2)\,,
 \end{multline}
 or equivalently,
 \begin{multline}
 \label{eq:62}
   \frac{rf''}{f'}=\frac{-\Big(
     |f'|^2-(p-3)\frac{f^2}{r^2}\Big)+
     \frac{f}{rf'}\Big(\frac{f^2}{r^2}-(p-3)|f'|^2\Big)}
   {{\frac{f^2}{r^2}}+(p-1)|f'|^2}\\      
 -\frac{2}{p}|\nabla
 u_p|^{2-p}\frac{f}{rf'}r^2(1-f^2)\cdot\frac{1}{1+(p-2)\frac{|f'|^2}{|\nabla u_p|^2}}\,.
 \end{multline}
 Put 
\begin{equation}
\label{eq:1000}
h = \frac{rf^\prime}{f}\,.
\end{equation}
 We divide the rest of the proof into several steps.\\[2mm]
\underline{Step 1:} $-\frac{1}{p-1}<h(r)<1$  for all $r>0\,.$ \\[2mm]
 We can rewrite \eqref{eq:62} as
 \begin{equation}
 \label{eq:9}
 r \frac{f''}{f^\prime} = \frac{-h^2+(p-3)+h^{-1}
   -(p-3)h}{1+(p-1)h^2}  -
   \frac{2}{p}|\nabla u_p|^{2-p}\frac{r^2}{h}(1-f^2)\frac{1+h^2}{1+(p-1)h^2}\,.
 \end{equation}
 Since 
 \begin{equation}
 \label{eq:63}
   h'=\frac{f''h}{f'}+\frac{h}{r}(1-h)\,,
 \end{equation}
 substituting \eqref{eq:9} into \eqref{eq:63} yields
 \begin{multline}
   \label{eq:32}
 h^{\prime} =\Big(\frac{1-h}{r}\Big) \cdot \Big(
 \frac{1+(p-2)h+h^2}{1+(p-1)h^2}+h\Big)-\frac{2}{p}|\nabla
 u_p|^{2-p}r(1-f^2)\frac{1+h^2}{1+(p-1)h^2}     \\  =      \frac{1+h^2}{1+(p-1)h^2} \left[\frac{(1-h)[1+(p-1)h]}{r} -
   \frac{2}{p}|\nabla u_p|^{2-p}r(1-f^2) \right] \,.
 \end{multline}
By \eqref{eq:32} we have
  \begin{equation}
\label{eq:66}
    h^\prime \leq \frac{1}{r}F_p(h) \,,
  \end{equation}
where
\begin{equation}
\label{eq:68}
  F_p(h) = \frac{(1+h^2)(1-h)[1+(p-1)h]}{1+(p-1)h^2}  \,.
\end{equation}
We now prove that $h(r)<1$ for all $r>0$. Suppose to the contrary that there  exists $r_0>0$
for which  $h(r_0)\geq1$. Then, \eqref{eq:66} yields $h^\prime(r)<0$ and
$h(r)> 1$ for all $r<r_0$. Therefore, by \eqref{eq:68} also $F_p(h)
<0$ for $r<r_0$. Integrating \eqref{eq:66} gives
\begin{equation}
\label{eq:67}
  \int_{h(r_0)}^{h(r)} \frac{dh}{-F_p(h)} \geq \ln \frac{r_0}{r} \,,\quad \forall r<r_0
  \,.
\end{equation}
Since $\int_{h(r_0)}^{\infty} \frac{dh}{-F_p(h)}<\infty$,  
\eqref{eq:67} leads to a contradiction for $r>0$ small
enough. 

Finally, we show that $h(r)>-\frac{1}{p-1}$ on $(0,\infty)$. Suppose to
the contrary that  $h(r_0)\leq -\frac{1}{p-1}$ for some $r_0$. Then, from
\eqref{eq:66} and  \eqref{eq:68} it follows that 
\begin{equation*}
   h(r)\leq -\frac{1}{p-1} \text{ and } h'(r)<0\,,\quad \forall r\geq r_0\,.
\end{equation*}
 Therefore, also $f_p^{'}(r)<0$ for all $r\geq r_0$, violating $I_p(f_p)<\infty$.
Step~1 is established.\\[2mm]
\underline{Step 2:}  $ \frac{f_p(r)}{r}$ is  strictly decreasing on
$(0,\infty)$.  \\[2mm]
From Step~1 we get that 
  \begin{equation}
\label{eq:34}
    \Big( \frac{f}{r} \Big)^\prime = \frac{f}{r^2}(h-1) < 0\,,~\forall r>0\,,
  \end{equation}
 and the conclusion follows.
\\[2mm]
\underline{Step 3:} $\lim_{r\to 0^+} h(r)=1.$\\[2mm]
Fix any $r_0>0$.  By Step~2 we have,
  \begin{displaymath}
    |\nabla u_p|(r)\geq\frac{f(r)}{r} > \frac{f(r_0)}{r_0}\,, \quad \forall r<r_0\,.
  \end{displaymath}
Consequently, we have by \eqref{eq:32},
\begin{equation}
\label{eq:33}
  h^\prime  \geq \frac{F_p(h)}{r} - C_0r\,,\quad \forall r\in(0,r_0)\,,
\end{equation}
for some positive $C_0$, which is independent of $r$. For
a contradiction, we assume that $\liminf_{r\to0^+} h(r)=a<1$. Then, using
\eqref{eq:33} we can find $r_1\in(0,r_0)$ small enough so that $h'(r_1)>0$.   
Bootstrapping we obtain that $h^\prime(r)>0$
for all $r<r_1$. In particular, the full limit  $\lim_{r\to0^+}
h(r)=a$ exists. Integration of \eqref{eq:33} then yields
\begin{equation}
 \label{eq:69}
 \int_{h(r)}^{h(r_1)} \frac{dh}{F_p(h)} \geq \ln \frac{r_1}{r} - C\,,
  \quad \forall r<r_1 \,.
\end{equation}
Here we used the fact that $F_p(h)>0$ by Step~1. Passing to the limit
$r\to0^+$ in \eqref{eq:69} gives  $\int_{a}^{h(r_1)}
\frac{dh}{F_p(h)}=\infty$. In view of \eqref{eq:68} we must have
\begin{equation*}
  a=\lim_{r\to0^+} h(r)=-\frac{1}{p-1}\,.
\end{equation*}
In particular, for $r$ sufficiently small we have
$\frac{rf^{'}}{f}\leq-\frac{1}{2(p-1)}$, implying
\begin{equation*}
  f(r)\geq Cr^{-\frac{1}{2(p-1)}}\,.
\end{equation*}
 A contradiction.\\[2mm]
\underline{Step 4:} $f'(0)$ exists and it is a positive number.\\[2mm]
By Step 2, the (possibly generalized) limit $\lim_{r\to 0^+}\frac{f(r)}{r}$
exists, so we only need to exclude the possibility that the limit
equals $+\infty$. From Step~3 and \eqref{eq:33} we get that
 $$
h(r)\geq1-cr^2 \,,~\forall r<r_0\,,
$$
i.e., 
\begin{displaymath}
  \frac{f^\prime}{f} \geq \frac{1}{r} - cr \,.
\end{displaymath}
Therefore, $f(r)\leq Cr$ for some positive constant $C$,
independently of $r$, and the differentiability of $f$ at $0$
follows. Finally, $f'(0)>0$ since $\frac{f(r)}{r}$ is decreasing.
\end{proof}

\subsection{Monotonicity}
\label{sec:3.4}
\begin{proposition}
\label{le2a}
  $f_p^{'}>0$ in $(0,\infty)$.
\end{proposition}
\begin{proof}
  First we show that $f_p$ is non-decreasing on $(0,\infty)$. 
Recall that $f_p'(0)>0$ and define
\begin{equation*}
  r_1=\sup \{ r\,:\,f_p'(s)\geq 0 \text{ on }[0,r]\}\,.
\end{equation*}
 If $r_1=\infty$ then clearly $f_p$ is non-decreasing on $(0,\infty)$. Assume
 then that $r_1<\infty$, and then obviously 
 \begin{equation*}
 f_p'(r_1)=0\,. 
 \end{equation*}
By the definition of $r_1$ we have also
\begin{equation}
  \label{eq:72}
f_p^{''}(r_1)\leq 0\,.
\end{equation}
Next we distinguish between two cases:\\[2mm]
{\rm (i)} There exists a right-neighborhood of $r_1$, $[r_1,R]$, in which
$f_p^{'}\leq 0$.\\[2mm]
{\rm (ii)}  There exists no neighborhood as in {\rm (i)}.\\[2mm]
Consider first case {\rm (i)}. Since
  $f_p\xrightarrow[r\to\infty]{}1$ , there must exist a maximal
  right-neighborhood,  where $f_p^\prime\leq0$ which we denote by
  $[r_1,r_2]$. Clearly, we must have
  $f_p^{'}(r_2)=0$. From \eqref{eq:91} we get that
  \begin{equation}
 \label{eq:65}
   \frac{r^2f_p^{\prime\prime}}{f_p}= 1 -
    \frac{2}{p}\Big(\frac{f_p}{r}\Big)^{-(p-2)}r^2
    (1-f_p^2)\,,\quad\text{for }r=r_i\,,~ i=1,2\,.
  \end{equation}
By Step~2 of the proof of Proposition~\ref{prop:at-zero} we have 
\begin{equation}
\label{eq:70}
  \Big(\frac{f_p(r_2)}{r_2}\Big)^{-(p-2)}> \Big(\frac{f_p(r_1)}{r_1}\Big)^{-(p-2)} \,.
\end{equation}
Furthermore, since $f_p^\prime\leq0$ in $[r_1,r_2]$ we have
\begin{equation}
\label{eq:71}
  (1-f^2_p)(r_2)\geq   (1-f^2_p)(r_1) \,.
\end{equation}
Substituting \eqref{eq:70},~\eqref{eq:71} into \eqref{eq:65} and using
\eqref{eq:72} yields
\begin{displaymath}
   \frac{r^2f_p^{\prime\prime}}{f_p}\bigg|_{r=r_2} <
   \frac{r^2f_p^{\prime\prime}}{f_p}\bigg|_{r=r_1} \leq 0 \,,
\end{displaymath}
i.e., $f_p^{\prime\prime}(r_2)<0$, which clearly contradicts the definition of $r_2$.

Next we turn to case {\rm (ii)}. In this case we have $f_p''(r_1)=0$. 
Differentiating the equation \eqref{eq:30} at $r=r_1$ yields
\begin{equation}
\label{eq:73}
  f_p^{(3)}(r_1)=-p\frac{f_p}{r_1^3}<0\,.
\end{equation}
 This implies that $r_1$ is a maximum point for $f_p'$ which is
 obviously impossible.

Finally, we prove that $f_p^{\prime}>0$ on $[0,\infty)$ (we know already that $f_p^{\prime}(0)>0$).
  Suppose, for a contradiction, that there exists $r_0>0$ such that
  \begin{displaymath}
    f_p^\prime(r_0)=f_p^{\prime\prime}(r_0)=0 \,.
  \end{displaymath}
We then obtain the same identity as in \eqref{eq:73}, but this time at
$r=r_0$. Again we get that $f_p^\prime$ has a maximum at $r_0$, a contradiction.
\end{proof}

To prove monotonicity of $f_p^\prime$ we need the following result
\begin{lemma}
  \label{lem:monotonicity2}
We have
\begin{equation}
\label{eq:10}
  h'\leq 0\,, \quad \forall r>0 \,.
\end{equation}
Furthermore, 
\begin{displaymath}
  \lim_{r\to\infty}h(r) = 0 \,.
\end{displaymath}
\end{lemma}
\begin{proof}
  Suppose, for a contradiction that \eqref{eq:10} does not hold. Since
  $\lim_{r\downarrow0}h(r)=1$ and $h<1$ on $(0,\infty)$ (see Steps~1 and 4 in  the proof of
  \rprop{prop:at-zero})  $h$ must have a minimum point at some
  $r=r_0$. By \eqref{eq:32} we have 
  \begin{multline}
\label{eq:93}
 h^{\prime\prime} (r_0) = -\frac{1}{r_0^2}F_p(h) -
 \frac{2}{p}\frac{1+h^2}{1+(p-1)h^2}|\nabla u_p|^{2-p}\bigg[
   -\frac{(|\nabla u_p|^2)^\prime}{|\nabla u_p|^2}\frac{p-2}{2}r_0(1-f_p^2) \\ +(1-f_p^2)
   -2r_0f_pf_p^\prime \bigg] \,. 
  \end{multline}
Furthermore, as $h^\prime(r_0)=0$ we also have  
  \begin{gather*}
   \frac{1}{r_0^2}F_p(h) = \frac{2}{p}
   \frac{1+h^2}{1+(p-1)h^2}|\nabla u_p|^{2-p}(1-f_p^2) \\
(|\nabla u_p|^2)^\prime \Big|_{r=r_0} = \bigg( \frac{f_p^2}{r^2}(1+h^2) \bigg)^\prime
\bigg|_{r=r_0} = -2(1+h^2)\frac{f_p}{r_0^2} \left(\frac{f_p}{r_0} -
  f^\prime_p\right) \,.
  \end{gather*}
Substituting the above into \eqref{eq:93} we obtain 
\begin{equation}
\label{eq:23}
 \text{sign}\, h^{\prime\prime}(r_0) = \text{sign}\,g(r_0)\,,
\end{equation}
where 
\begin{equation}
  \label{eq:24}
g(r):=2hf_p^2 - \{(p-2)(1-h)+2\}(1-f_p^2)\,.
\end{equation}
Since $r_0$ is a minimum point of $h$, we must have $g(r_0)\geq 0$.
Put
 \begin{equation*}
  r_1=\sup \{ r\in (r_0,\infty)\,:\,h'\geq 0 \text{ on }(r_0,r]\}\,.
\end{equation*}
If $r_1=\infty$ then, since $h<1$, $h\xrightarrow[r\to\infty]{}h_\infty$ where
$0<h_\infty\leq1$. But this leads to a contradiction since then also
\begin{displaymath}
  rf_p^\prime \xrightarrow[r\to\infty]{}h_\infty \,,
\end{displaymath}
 which is inconsistent with $\lim_{r\to\infty} f(r)=1$.
 If $r_1<\infty$ then necessarily $h'(r_1)=0$ and $h''(r_1)\leq 0$, implying
 that $g(r_1)\leq 0$ too. But since
 $h$ is non-decreasing on $(r_0,r_1)$ while $f$ is strictly increasing
 on $(r_0,r_1)$ (by \rprop{le2a}), it
 follows from \eqref{eq:24} that $g$ is strictly increasing on
 $(r_0,r_1)$. Therefore, $g(r_1)>g(r_0)\geq0$, implying as in
 \eqref{eq:23} that $h''(r_1)>0$. This contradiction completes the proof of \eqref{eq:10}.
 
Finally, as $h$ is both positive and decreasing it must converge to a
limit $h_\infty\geq 0$. From the above argument we obtain that $h_\infty=0$. 
\end{proof}

\begin{corollary}
\label{cor:f''<0}
  $f_p^\prime$ is monotone decreasing in $\R_+$.
\end{corollary}
The corollary follows immediately from the fact that $f_p^\prime$ is a
product of the positive functions $h$ and $f_p/r$, the first of which is
non-increasing, and the second is strictly decreasing.
 
\subsection{Asymptotic behavior}
\label{sec:3.5}

In the following we derive the behavior of $1-f_p^2$ as $r\to\infty$. The
first lemma is a well-established result in asymptotic analysis. We include the
proof for the convenience of the reader.
\begin{lemma}
\label{le3}
  Let $g(x)$ be monotone decreasing on $(0,\infty)$. Let further
  \begin{displaymath}
    \int_r^\infty g(t) dt = \frac{1}{r^\alpha}[1+o(1)] \quad \text{as } r\to\infty \,.
  \end{displaymath}
for some positive $\alpha$. Then,
\begin{displaymath}
   g(r)  = \frac{\alpha}{r^{\alpha+1}}[1+o(1)] \quad \text{as } r\to\infty \,.
\end{displaymath}
\end{lemma}
\begin{proof}
  Put $G(r)=\int_r^\infty g(t)\,dt.$ Then, for any $h>0$,
  \begin{equation}
\label{eq:105}
    hg(r)\geq \int_r^{r+h} g(t)\,dt=G(r)-G(r+h)=\frac{1+\eta(r)}{r^\alpha}-\frac{1+\eta(r+h)}{(r+h)^\alpha}\,,
  \end{equation}
 where $\lim_{r\to\infty} \eta(r)=0$. By  \eqref{eq:105},
 \begin{align*}
  hg(r)&\geq
  (1+\eta(r))\Big(\frac{1}{r^\alpha}-\frac{1}{(r+h)^\alpha}\Big)+\frac{\eta(r)-\eta(r+h)}{r^\alpha}\\&\geq\frac{1}{r^\alpha} \Big( (1-\eta_m)\big\{1-(1+\frac{h}{r})^{-\alpha}\big\}-2\eta_m\Big)\,,
 \end{align*}
 where $\eta_m(r,h)=\max( |\eta(r)|,|\eta(r+h)| )$. Let $\epsilon=\frac{h}{r}$. Since
 for some $C>0$ we have
 \begin{equation*}
   1-(1+\epsilon)^{-\alpha}\geq 1+\alpha\epsilon-C\epsilon^2\,,~\epsilon\in[0,\frac{1}{2}]\,,
 \end{equation*}
it follows that 
\begin{equation*}
  hg(r)\geq \frac{1}{r^\alpha} \Big((1-\eta_m)(\alpha\epsilon-C\epsilon^2)-2\eta_m\Big)\,.
\end{equation*}
 Therefore,
 \begin{equation}
   \label{eq:106}
g(r)\geq \frac{1}{r^{\alpha+1}}\Big((1-\eta_m)(\alpha-C\epsilon) -2\frac{\eta_m}{\epsilon}\Big)\,.
 \end{equation}
 Choosing $\epsilon=\eta_m^{1/2}$  we get from \eqref{eq:106} (since
 $\lim_{r\to\infty} \sup_{h>0}\eta_m(r,h)=0$),
 \begin{equation*}
   g(r)\geq \frac{\alpha}{r^{\alpha+1}}(1+o(1))\,,~\text{ as }r\to\infty\,.
 \end{equation*}
The second direction is proved in a similar manner.
\end{proof}
We use the above lemma to prove the following
result
\begin{lemma}
\label{le5}
  \begin{subequations}
\label{eq:1}
      \begin{gather}
    1-f_p^2\sim\frac{p}{2}\frac{1}{r^p} \quad \text{as } r\to\infty \,, \\
    f_p^\prime \sim \frac{p^2}{4}\frac{1}{r^{p+1}} \,.
  \end{gather}
  \end{subequations}
\end{lemma}

\begin{proof}
  Integrating by parts \eqref{eq:30} between $r$ and infinity yields
  \begin{equation*}
\int_r^\infty f_p(1-f_p^2) dt = \frac{p}{2} \int_r^\infty |\nabla u_p|^{p-2} \Big[
\frac{f_p}{t} - f_p^\prime\Big] \frac{dt}{t} + \frac{p}{2} |\nabla u_p|^{p-2}f_p^\prime \,,
  \end{equation*}
or equivalently that
\begin{equation}
 \label{eq:35}
  \int_r^\infty f_p(1-f_p^2) dt = \frac{p}{2} \int_r^\infty |1+h^2|^{(p-2)/2} (1-h)
\bigg(\frac{f_p}{t}\bigg)^{p-1} \frac{dt}{t} + \frac{p}{2}
|1+h^2|^{(p-2)/2} h \bigg(\frac{f_p}{r}\bigg)^{p-1} \,.
\end{equation}

Applying the integral mean value theorem yields the existence of
$r^*\in[r,\infty)$ such that
\begin{displaymath}
  \int_r^\infty |1+h^2|^{(p-2)/2} (1-h)
\bigg(\frac{f_p}{t}\bigg)^{p-1} \frac{dt}{t}=    |1+h^2(r^*)|^{(p-2)/2}
  (1-h(r^*))f_p^{p-1}(r^*)\frac{r^{-(p-1)}}{p-1}\,. 
\end{displaymath}
Hence, in view of Lemma \ref{lem:monotonicity2} and the fact that
$f_p\xrightarrow[r\to\infty]{}1$ we obtain
\begin{equation}
\label{eq:4}
  \int_r^\infty |1+h^2|^{(p-2)/2} (1-h)
\bigg(\frac{f_p}{t}\bigg)^{p-1} \frac{dt}{t}= 
   \frac{r^{-(p-1)}}{p-1}[1+o(1)] \quad \text{as } r\to\infty\,.
\end{equation}
Further, in view of Lemma \ref{lem:monotonicity2} we have
\begin{equation}
\label{eq:5}
  |1+h^2|^{(p-2)/2} h \bigg(\frac{f_p}{r}\bigg)^{p-1} = o(r^{-(p-1)}) \,.
\end{equation}
Substituting \eqref{eq:4}--\eqref{eq:5} into \eqref{eq:35} yields
\begin{displaymath}
  \int_r^\infty f_p(1-f_p^2) dt = \frac{p}{2} \frac{r^{-(p-1)}}{p-1}[1+o(1)]
  \quad \text{as } r\to\infty \,.
\end{displaymath}
As $f_p\xrightarrow[r\to\infty]{}1$ we have 
\begin{displaymath}
    \int_r^\infty f_p(1-f_p^2) dt =   [1+o(1)] \int_r^\infty (1-f_p^2) dt
  \quad \text{as } r\to\infty \,,
\end{displaymath}
and hence
\begin{displaymath}
  \int_r^\infty (1-f_p^2) dt  = \frac{p}{2} \frac{r^{-(p-1)}}{p-1}[1+o(1)]
  \quad \text{as } r\to\infty \,.
\end{displaymath}
The proof of (\ref{eq:1}a) follows immediately from Lemma \ref{le3}
and the monotonicity of $f_p$. 

To prove (\ref{eq:1}b) we first note that
\begin{displaymath}
  \lim_{r\to\infty}\frac{1-f_p^2}{1-f_p}=2 \,.
\end{displaymath}
Hence,
\begin{displaymath}
  \int_r^\infty f_p^\prime dt  = \frac{p}{4r^p} [1+o(1)]
  \quad \text{as } r\to\infty \,.
\end{displaymath}
Lemma \ref{le3} provides, once again, the closing argument for the
proof.
\end{proof}

\section{Large $p$}
\label{sec:large}
In this section we discuss the behavior of the radially symmetric
solution in the large $p$ limit. We prove the following result
\begin{theorem}
\label{thplarge}
  Let
  \begin{equation}
    \label{eq:41}
f_\infty =
\begin{cases}
  \frac{r}{\sqrt{2}} & r<\sqrt{2} \\
  1 & r \geq \sqrt{2}
\end{cases}\,.
  \end{equation}
There exists $C>0$ such that for every $p>2$ we have
\begin{equation}
  \label{eq:42}
\| f_p - f_\infty\|_{L^\infty(\R_+)}=\| f_p - f_\infty\|_\infty \leq C\Big(\frac{\ln p}{p}\Big)^{1/2}\,.
\end{equation}
\end{theorem}

To prove the theorem we shall need to prove first a few auxiliary
results. We first derive a simple upper bound
\begin{lemma}
\label{lem:up}
We have
  \begin{equation}
\label{eq:14}
     I_p(f_p) \leq \left( \frac{1}{6} + C \frac{\ln p}{p} \right)\,,
  \quad \forall p>2\,.
  \end{equation}
\end{lemma}
\begin{proof}
  We use the test function
\begin{displaymath}
  \tilde{f}=
  \begin{cases}
    \frac{1}{\sqrt{2}}\Big(1-\frac{\ln p}{p}\Big)r\,, &
      r<\frac{\sqrt{2}}{1-\frac{\ln p}{p}} \\
      1\,, &  r \geq \frac{\sqrt{2}}{1-\frac{\ln p}{p}}
  \end{cases}\,.
\end{displaymath}
It is easy to show that there exists $C>0$, independent of $p$ such that 
\begin{displaymath}
  I_p(\tilde{f}) \leq \left( \frac{1}{6} + C \frac{\ln p}{p} \right)\,,
  \quad \forall p>2\,,
\end{displaymath}
from which the lemma immediately follows. 
\end{proof}

We first deal with the interval $[0,\sqrt{2}]$.
\begin{proposition}
\label{prop:Linfty}
  We have
  \begin{equation}
    \label{eq:43}
\exists C>0:\; \|\nabla u_p\|_\infty \leq 1 + \frac{C}{p}\,, \quad \forall p>2 \,.
  \end{equation}
\end{proposition}
\begin{proof}
We first note that by \rlemma{lem:monotonicity2} and Step~2
of the proof of  \rprop{prop:at-zero} both $f_p^\prime$ and
$f_p/r$ are decreasing. Therefore, the same holds for  $|\nabla u_p|$
and it follows that
\begin{equation}
  \label{eq:25}
 \|\nabla u_p\|_\infty=|\nabla u_p(0)|\,.
\end{equation}
Obviously, if we have
$|\nabla u_p|-1\gg1/p$ over a sufficiently large right semi-neighborhood of
$r=0$, then $I_p(f)$ would become larger than the upper bound \eqref{eq:14}. This, however, does not eliminate
the possibility of a small neighborhood of $r=0$ where $p(|\nabla u_p|-1)$ is
large. Thus, the proof splits into two parts: at first, using
regularity arguments, we bound from below the size of the above
neighborhood as a function of $|\nabla u_p(0)|$. Then, we use \eqref{eq:14}
to bound $|\nabla u_p(0)|$ from above.

  Suppose that $ |\nabla u_p(0)| =a>1$.  Let
  \begin{equation}
\label{eq:26}
 s=\sup \bigg\{ r>0\,:\, |\nabla u_p(r)|>\frac{1+a}{2}  \bigg\} \,.
  \end{equation}
By \eqref{eq:32} we have for all $r<s$ that
\begin{equation}
\label{eq:19}
  h^\prime \geq -\frac{2}{p}\bigg(\frac{1+a}{2}\bigg)^{-(p-2)}r \,,
\end{equation}
hence
\begin{equation}
\label{eq:16}
  1-\frac{1}{p}\Big(\frac{1+a}{2}\Big)^{-(p-2)}r^2\leq h\leq 1\,, \quad \forall r\leq s\,.
\end{equation}
Assume first that 
\begin{equation}
  \label{eq:17}
s^2\leq \frac{p}{2} \Big(\frac{1+a}{2}\Big)^{p-2}\,,
\end{equation}
 implying by \eqref{eq:16} that
 \begin{equation}
   \label{eq:18}
h\geq \frac12\,,\quad \forall r\leq s\,.
 \end{equation}
By \eqref{eq:63} we have
\begin{equation}
  \label{eq:15}
\frac{f_p^{''}}{f_p'}=\frac{h'}{h}-\frac{1-h}{r}\,.
\end{equation}
Therefore, using \eqref{eq:18} and \eqref{eq:19}--\eqref{eq:16} we deduce that
\begin{displaymath}
  \bigg|\frac{f_p^{\prime\prime}}{f_p^\prime}\bigg| =-\frac{f_p^{\prime\prime}}{f_p^\prime}\leq
  \frac{C}{p}\bigg(\frac{1+a}{2}\bigg)^{2-p}r \,,\quad \forall r\leq s\,,
\end{displaymath}
implying
\begin{equation}
\label{eq:20}
   \exp\Big\{ -\frac{C}{2p}\bigg(\frac{1+a}{2}\bigg)^{2-p}r^2\Big\} \leq
   \frac{f_p^\prime(r)}{f_p^\prime(0)}  \,,\quad \forall r\leq s\,.
\end{equation}
Since $h<1$,
\begin{displaymath}
  \frac{|\nabla u_p(r)|^2}{|\nabla u_p(0)|^2} =
  \frac{(1+h^{-2})|f_p^\prime(r)|^2}{2|f_p^\prime(0)|^2} \geq \frac{|f_p^\prime(r)|^2}{|f_p^\prime(0)|^2}\,,\quad \forall r\leq s\,.
\end{displaymath}
Consequently, by \eqref{eq:20}
\begin{equation}
\label{eq:47}
  \exp\Big\{ -\frac{C}{p}\bigg(\frac{1+a}{2}\bigg)^{2-p}r^2\Big\} \leq
  \frac{|\nabla u_p(r)|^2}{|\nabla u_p(0)|^2} \,,\quad \forall r\leq s\,.
\end{equation}
Setting $r=s$ in \eqref{eq:47} we obtain
\begin{equation*}
s^2 \geq Cp\bigg(\frac{1+a}{2}\bigg)^{p-2} \ln \bigg(\frac{2a}{1+a}\bigg) \geq
Cp \bigg(\frac{1+a}{2}\bigg)^{p-2}\frac{a-1}{a}\,.
\end{equation*}

If \eqref{eq:17} doesn't hold, then clearly
\begin{equation*}
  s^2>  \frac{p}{2} \Big(\frac{1+a}{2}\Big)^{p-2}\,.
\end{equation*}
 Therefore, in all cases we have
 \begin{equation}
   \label{eq:21}
 s^2\geq Cp\frac{a-1}{a}\bigg(\frac{1+a}{2}\bigg)^{p-2}\,.
 \end{equation}
To conclude, we shall use the upper-bound for the energy from
\rlemma{lem:up} in order to
bound $s$ from above. Combining \eqref{eq:21} with \eqref{eq:14} and
\eqref{eq:26} yields
\begin{equation}
  \label{eq:45}
C\geq \int_0^s |\nabla u_p|^prdr \geq \frac{s^2}{2}{\bigg(\frac{1+a}{2}\bigg)}^p\geq
Cp\frac{a-1}{a}\bigg(\frac{1+a}{2}\bigg)^{2(p-1)}\geq  Cp(a-1)\,,
\end{equation}
 From \eqref{eq:45} we get
\begin{displaymath}
  a \leq 1+ \frac{C}{p} \,,\quad \forall p>2 \,,
\end{displaymath}
and \eqref{eq:43} follows from \eqref{eq:25}.
\end{proof}

We can now obtain $L^\infty$ convergence of $f_p$ to $f_\infty$ in every compact
set in $[0,\sqrt{2})$.
\begin{proposition}
For every $b\in(0,\sqrt{2})$ there exists $C=C(b)>0$ such that,
\begin{subequations}
   \label{eq:48}
   \begin{align}
&\Big\| f_p -\frac{r}{\sqrt{2}}\Big\| _{L^\infty(0,b)} \leq C \frac{\ln
  p}{p}\,, \quad p>2,\\
& \Big\| f_p^\prime -\frac{1}{\sqrt{2}}\Big\| _{L^\infty(0,b)} \leq C \frac{\ln
  p}{p}\,,  \quad p>2  \,.
 \end{align}
\end{subequations}
\end{proposition}
\begin{proof}
First we note that by \eqref{eq:43}
\begin{equation*}
  |\nabla u_p(0)|^2=2f'_p(0)^2\leq 1+\frac{C}{p}~\Longrightarrow~f'_p(0)\leq \frac{1}{\sqrt{2}}+\frac{C}{p}\,.
\end{equation*} 
Since $f_p'$ is decreasing, we conclude that
\begin{equation}
  \label{eq:28}
f'_p(r)\leq \frac{1}{\sqrt{2}}+\frac{C}{p}\,.
\end{equation}
Integrating \eqref{eq:28}, using  $f_p(0)=0$, yields the existence of  $C>0$ such that for
every $p>2$ we have, for all $r>0$,
\begin{equation}
\label{eq:99}
  f_p (r)\leq \frac{1}{\sqrt{2}}\Big(1+\frac{C}{p}\Big)r \,.
\end{equation}
 Put
 \begin{equation*}
   w(r)=\frac{r}{\sqrt{2}}-f_p(r).
 \end{equation*}
By \eqref{eq:28}--\eqref{eq:99} we have
\begin{equation*}
  w'(r)\geq -\frac{C}{p}\geq -C\frac{\ln p}{p}\quad\text{and}\quad  w(r)\geq
  -\frac{C}{p}\geq  -C\frac{\ln
  p}{p}\,,\quad\forall r>0\,.
\end{equation*}
In order to conclude, we need to prove that for each $b\in(0,\sqrt{2})$ there exists $C_b$ such that
 \begin{equation}
   \label{eq:31}
w'(r)\leq C_b\frac{\ln p}{p}\quad\text{and}\quad  w(r)\leq C_b\frac{\ln
  p}{p}\,,\quad \forall r\in[0,b]\,.
 \end{equation}
For such $b$ we set $\tilde b=\frac{b+\sqrt{2}}{2}$ and
claim that
 \begin{equation}
   \label{eq:36}
\int_b^{\tilde b} w(r)\,dr \leq  C_b\frac{\ln p}{p}\,.
 \end{equation}
To prove \eqref{eq:36} we first note that
\begin{multline*}
  \frac{1}{2}\int_0^\infty \big(1-f_p^2\big)^2 r\,dr\geq  \frac{1}{2}\int_0^{\sqrt{2}}
  \big(1-f_p^2\big)^2  r\,dr= \frac{1}{2} \int_0^{\sqrt{2}}
  \bigg[1-\frac{1}{2}r^2\bigg]^2rdr + \\
\int_0^{\sqrt{2}}   \bigg[1- \frac{1}{2}r^2\bigg]
\bigg[\frac{1}{2}r^2-f_p^2\bigg]r\,dr+ \frac{1}{2} \int_0^{\sqrt{2}} 
\bigg[\frac{1}{2}r^2-f_p^2\bigg]^2r\,dr\,.
\end{multline*}
Since 
\begin{equation*}
  \frac{1}{2} \int_0^{\sqrt{2}} \bigg[1-\frac{1}{2}r^2\bigg]^2rdr=\frac{1}{6}\,,
\end{equation*}
 we deduce, using \eqref{eq:14}, that 
 \begin{equation*}
   \int_0^{\sqrt{2}}   \bigg[1- \frac{1}{2}r^2\bigg]
\bigg[\frac{1}{2}r^2-f_p^2\bigg]r\,dr\leq  C\frac{\ln p}{p}\,.
 \end{equation*}
Therefore 
\begin{equation*}
 C_b\frac{\ln p}{p}\geq  (1-\frac{{\tilde b}^2}{2})\int_b^{\tilde b} w(r)(\frac{r}{\sqrt{2}}+f_p)\,rdr\,,
\end{equation*}
 and \eqref{eq:36} follows.
 Finally, using the convexity of $w$ in conjunction with \eqref{eq:36}
 gives
 \begin{equation*}
   w(b)(\tilde b-b)+\frac{w'(b)}{2}(\tilde b-b)^2=\int_b^{\tilde b} (w(b)+(r-b)w'(b))\,dr \leq\int_b^{\tilde b} w(r)\,dr \leq C_b\frac{\ln p}{p}\,,
 \end{equation*}
 implying, in particular, that
 \begin{equation}
   \label{eq:44}
 w'(b)\leq C_b\frac{\ln p}{p}\,.
 \end{equation}
 Since $w'$ is increasing we deduce the first inequality in
 \eqref{eq:31}. The second one follows by integration of the first one.
\end{proof}

We now improve the estimates \eqref{eq:48}. We start by deriving a
Pohozaev-type identity.
\begin{proposition}
We have
\begin{equation}
\label{eq:57}
  \int_0^\infty |\nabla u_p|^p r\,dr = \frac{2}{p}m_p \,,
\end{equation}
where $m_p$ is defined  in \eqref{eq:94}.
\end{proposition}
\begin{proof}
  Let  $f_p^{(\alpha)}(r) = f_p(\alpha r)$ and $J=  \int_0^\infty |\nabla u_p|^p r\,dr $. Clearly,
  \begin{displaymath}
    M_\alpha = I_p(f_p^{(\alpha)})= \alpha^{p-2}J + \frac{1}{\alpha^2}(m_p-J) \,.
  \end{displaymath}
Hence,
\begin{displaymath}
  \frac{dM_\alpha}{d\alpha} = (p-2)\alpha^{p-3}J - \frac{2}{\alpha^3}(m_p-J) \,.
\end{displaymath}
Since  $M_\alpha$ must have a global minimum at $\alpha=1$, \eqref{eq:57} follows.
\end{proof}
\begin{corollary}
  We have
  \begin{equation}
    \label{eq:58}
\liminf_{p\to\infty}p|\nabla u_p(0)|^p \geq \frac{1}{3} \,.
  \end{equation}
\end{corollary}
\begin{proof}
Since $f_p^\prime<f_p/r<1/r$ we have 
\begin{displaymath}
  |\nabla u_p| \leq \frac{\sqrt{2}}{r} \,.
\end{displaymath}
Thus, for every $l>0$,
\begin{equation*}
  \int_{l}^\infty |\nabla u_p|^p r dr \leq \frac{{2}^{p/2}}{(p-2)l^{p-2}}\,,
\end{equation*}
from which we get
\begin{equation}
  \label{eq:59}
\lim_{p\to\infty}\int_l^\infty |\nabla u_p|^p\,rdr=0\,,\quad \forall l>\sqrt{2}\,.
\end{equation}

By \eqref{eq:48} we have 
  \begin{displaymath}
    \liminf_{p\to\infty}m_p \geq  \sup_{b\in (0,\sqrt{2})}\liminf_{p\to\infty}\frac{1}{2}\int_0^{b}
    (1-f_p^2)^2 rdr = \frac{1}{6} \,.
  \end{displaymath}
Thus, by \eqref{eq:59} and \eqref{eq:57} we have for all $l> \sqrt{2}$,
\begin{displaymath}
  \liminf_{p\to\infty} p \int_0^l |\nabla u_p|^p rdr \geq  \frac{1}{3} \,.
\end{displaymath}
As $|\nabla u_p(r)|\leq |\nabla u_p(0)|$ we deduce that
\begin{displaymath}
   \liminf_{p\to\infty} p|\nabla u_p(0)|^p\frac{l^2}{2} \geq \frac{1}{3} \quad \forall l>\sqrt{2}\,,  
\end{displaymath}
from which \eqref{eq:58} readily follows.
\end{proof}

\begin{lemma}
  Let $g=|\nabla u_p|^p$, and
  \begin{displaymath}
    g_0 = \frac{1}{p}\Big(1-\frac{1}{2}r^2\Big)^2 \,.
  \end{displaymath}
Then,
\begin{equation}
  \label{eq:52}
\lim_{p\to\infty}p\|g-g_0\|_{L^\infty(0,a)} =0\,, \quad \forall a<\sqrt{2} \,.
\end{equation}
\end{lemma}
\begin{proof}
  Multiplying \eqref{eq:30} by $rf_p$ and integrating over $[0,r]$, we obtain
  \begin{subequations}
    \label{eq:53}
      \begin{equation}
\frac{p}{4}r^2g(1-\alpha_p) -\frac{p}{2}\int_0^rg(t)tdt + h(r) = 0 \,,
  \end{equation}
in which
\begin{gather}
  \alpha_p = 1 - \frac{2\frac{f_p}{r}f_p^\prime}{|\nabla u_p|^2}>0 \,,
\intertext{and}
h(r) = \int_0^rf_p^2(1-f_p^2)tdt\,.
\end{gather}
 \end{subequations}
We may write
\begin{equation*}
  h(r)=h_0(r)(1+\beta_p)~\text{ with }~h_0(r) = \int_0^r \frac{t^2}{2}\Big(1-\frac{t^2}{2}\Big)t\,dt=\frac{1}{8}\Big(r^4-\frac{1}{3}r^6\Big)\,. 
\end{equation*}
Set
\begin{equation}
\label{eq:60}
\epsilon_p(a) = \max \big( \|\alpha_p\|_{L^\infty(0,a)} , \|\beta_p\|_{L^\infty(0,a)} \big)\,.
\end{equation}
By \eqref{eq:48} there exists $C>0$ such that
\begin{equation}
\label{eq:54}
  \epsilon_p(a) \leq C \frac{\ln p}{p}\,,
\end{equation}
for all fixed $a<\sqrt{2}$. 

Set
\begin{displaymath}
  G(r) = \int_0^r g(t)t\,dt\,,
\end{displaymath}
to obtain from \eqref{eq:53} that
\begin{equation}
\label{eq:2}
  G^\prime- \gamma_pG + \frac{2}{p}\gamma_ph = 0\,,
\end{equation}
where
\begin{displaymath}
  \gamma_p = \frac{2}{r(1-\alpha_p) }\,.
\end{displaymath}
Solving \eqref{eq:2} and then evaluating $G'$ once again from
\eqref{eq:2} yields  the general solution of \eqref{eq:53}: 
  \begin{equation}
\label{eq:55}
    g(r)= -\frac{2}{p}\frac{\gamma_p}{r} \left[ h + \int_0^r \exp \bigg\{ \int_t^r \gamma_p(s)ds
      \bigg\}\gamma_p(t) h(t)dt + C_0 \exp \bigg\{ -\int_r^a \gamma_p(t)dt \bigg\} \right]\,,
  \end{equation}
where $C_0$ is arbitrary.

 First we compute
 \begin{displaymath}
   \frac{\gamma_p}{r} \exp \bigg\{ -\int_r^a
   \gamma_p(t)dt \bigg\} \leq  \frac{\gamma_p}{r} \exp \bigg\{ -2\int_r^a
  \frac{dt}{t} \bigg\}=\frac{\gamma_pr}{a^2}\,. 
 \end{displaymath}
On the other hand, a similar computation gives 
\begin{displaymath}
   \frac{\gamma_p}{r} \exp \bigg\{ -\int_r^a \gamma_p(t)dt \bigg\}\geq  \frac{\gamma_p}{r} \bigg(\frac{r}{a}\bigg)^{\frac{2}{1-\epsilon_p}}\,.
\end{displaymath}
Therefore,
\begin{displaymath}
   \frac{2}{a^2}\bigg(\frac{r}{a}\bigg)^{2\big((1-\epsilon_p)^{-1}
     -1\big)}\leq \frac{\gamma_p}{r} \exp \bigg\{ -\int_r^a
   \gamma_p(t)dt \bigg\} 
\leq \frac{2}{a^2}\big(1+C\epsilon_p \big) \,.
\end{displaymath}
 Similarly,
\begin{multline*}
  \frac{\gamma_p}{r}\int_0^r \exp \bigg\{ \int_t^r \gamma_p(s)\,ds
      \bigg\} \gamma_p(t)h(t)\,dt \geq  \frac{\gamma_p}{r}\int_0^r
      \bigg(\frac{r}{t}\bigg)^2 \gamma_p(t)h(t)\,dt\\
 \geq 4\int_0^r \frac{h(t)}{t^3}\,dt\geq \Big(\frac{1}{4}r^2 -
      \frac{1}{24}r^4\Big)\big( 1-\epsilon_p\big) \,,
\end{multline*}
 and 
 \begin{displaymath}
    \frac{\gamma_p}{r}\int_0^r \exp \bigg\{ \int_t^r \gamma_p(s)\,ds
      \bigg\} \gamma_p(t)h(t)\,dt \leq \Big(\frac{1}{4}r^2 -
      \frac{1}{24}r^4\Big)\big( 1+C\epsilon_p\big)\,.
 \end{displaymath}

Combining the above with \eqref{eq:55} we obtain that
\begin{equation}
  \label{eq:56}
\frac{2}{p}\bigg[\tilde{C}_0r^{2\big((1-\epsilon_p)^{-1} -1\big)}
-\frac{1}{2}r^2 + \frac{1}{8}r^4-C\epsilon_p\bigg] 
 \leq g \leq
\frac{2}{p}\bigg[\tilde{C}_0 -\frac{1}{2}r^2 + \frac{1}{8}r^4+C\epsilon_p\bigg]  \,.
\end{equation}
Note that the above lower bound is unsatisfactory in some neighborhood of
$r=0$ where
\begin{displaymath}
  1-r^{2\big((1-\epsilon_p)^{-1} -1\big)}\sim \OO( \epsilon_p)\,,
\end{displaymath}
which is valid for $r\sim\OO(1)$ as $p\to\infty$. 

We defer the proof of convergence near $r=0$ to a later stage and
instead prove first the existence of
$\lim_{p\to\infty}\tilde{C}_0(p)$, and then obtain its value.  Clearly,
\begin{displaymath}
  \liminf_{p\to\infty}\tilde{C}_0(p)\geq \frac{1}{2},
\end{displaymath}
otherwise $g$ would become negative, for some sufficiently large $p$ and
a fixed $r_0<\sqrt{2}$ - a contradiction.  Suppose now to the
contrary, that a sequence $\{p_k\}_{k=1}^\infty$ exists such that
$\tilde{C}_0(p_k)=C_k\to b$, where $b\in(\frac{1}{2},\infty]$.  By (\ref{eq:43})
we have  
\begin{displaymath}
  \|g(\cdot,p_k)\|_{L^\infty(\R_+)} \leq C\,,
\end{displaymath}
where $C$ is independent of $k$. Hence, by (\ref{eq:56}) we have
\begin{equation}
\label{eq:50}
  C_k \leq Cp_k \,.
\end{equation}

Set
\begin{displaymath}
  g_{0,k} =  2\bigg[C_k -\frac{1}{2}r^2 +
  \frac{1}{8}r^4\bigg] \,.
\end{displaymath}
Note that by our supposition $\lim g_{0,k}(r)>0$ in $[0,\sqrt{2}+\delta]$ for
some $\delta>0$.
It follows from (\ref{eq:56}) and (\ref{eq:50}) that 
\begin{equation}
\label{eq:49}
 \frac{\ln (g_{0,k}-\epsilon_k)}{p_k} -2 \bigg|\frac{\ln (g_{0,k}-\epsilon_k)}{p_k}\bigg|^2 \leq  |\nabla u_p| -1 +
 \frac{\ln p_k}{p_k} \leq  \frac{\ln (g_{0,k}+\epsilon_k)}{p_k}+ 2 \bigg|\frac{\ln (g_{0,k}+\epsilon_k)}{p_k}\bigg|^2
\end{equation}
where $\epsilon_k(a)=\epsilon(p_k)(a)$. 

We argue from here by bootstrapping. Let $a\in(0,\sqrt{2}+\delta]$ be such
that 
\begin{equation}
\label{eq:51}
  \limsup \epsilon_k(a)\leq\limsup \frac{g_{0,k}(a)}{2}\,.
\end{equation}
For sufficiently large $k$ we have, in view of (\ref{eq:49}) and
(\ref{eq:50}) and the fact that $\epsilon_k(\sqrt{2}+\delta)$ is bounded , that
\begin{displaymath}
  2\frac{f_k}{r}f_k^\prime \leq |\nabla u_p|^2  \leq 1+ \frac{C}{p_k} \quad
  \forall r\in[0,\sqrt{2}+\delta]\,,
\end{displaymath}
where $f_k=f_{p_k}$, from which we obtain that
\begin{equation}
\label{eq:11}
  \frac{f_k}{r}  \leq \frac{1}{\sqrt{2}}+ \frac{C}{p_k} \quad
  \forall r\in[0,\sqrt{2}+\delta]\ \,.
\end{equation}
Consequently, by (\ref{eq:49}) and (\ref{eq:51}), we have for sufficiently large $k$ that
\begin{equation}
\label{eq:12}
  f_k^\prime \geq \frac{1}{\sqrt{2}} - C\frac{\ln p_k}{p_k}  \quad \forall r\in[0,a]\,,
\end{equation}
where $C$ is independent of $a$.
Since $f_k^\prime\leq f_k/r$, we have by (\ref{eq:53}b), for sufficiently large $k$, that 
\begin{equation}
\label{eq:61}
   \alpha_{p_k}\leq  \frac{\ln p_k}{p_k}  \quad \forall r\in[0,a]\,,
\end{equation}
where $C$ is independent of $a$. Furthermore, by (\ref{eq:53}c,d),
(\ref{eq:11}), (\ref{eq:12}), and the fact that $f_k/r>f_k^\prime$ there
exists $C>0$ which is independent of both $k$ and $a$ such that
\begin{displaymath}
  \beta_{p_k}(r) \leq C\frac{\ln p_k}{p_k} 
\end{displaymath}
for all $r\leq a$.  Combining the above and (\ref{eq:61}) we obtain for
sufficiently large $k$
\begin{equation}
\label{eq:74}
  \limsup \epsilon_k(a)\leq\lim \frac{g_{0,k}(a)}{2} \Rightarrow \limsup \,\frac{p_k}{\ln
    p_k}\epsilon_k \leq C  \,,
\end{equation}
where $C$ is independent of $a$. From (\ref{eq:48}) we thus have
$$\limsup \frac{p_k \epsilon_k}{\ln p_k} \leq C$$ 
for all $a<\sqrt{2}$.

Let then $a_0$ be such that  
\begin{displaymath}
  \limsup \epsilon_k(a_0)=\lim \frac{g_{0,k}(a_0)}{2}\,.
\end{displaymath}
Since by (\ref{eq:74}) we have $\lim g_{0,k}(a_0)=0$, it follows that
$a_0>\sqrt{2}+\delta$. Hence
\begin{displaymath}
  \limsup \frac{p_k}{\ln p_k} \epsilon_k(\sqrt{2}+\delta) \leq C \,.
\end{displaymath}
Substituting into (\ref{eq:56}) we obtain that
\begin{displaymath}
  \lim \frac{p_k}{\ln p_k} g(\sqrt{2}+\delta) >0 \,.
\end{displaymath}
Let $l> \sqrt{2}$. Then, $f_p^\prime(l)<f_p(l)/l<1/l$, and hence
$g(l)\leq(\sqrt{2}/l)^p$. Consequently, $g(l)$ is exponentially small for
all $l>\sqrt{2}$ as $p\to\infty$, and in particular at $l=\sqrt{2}+\delta$ - a contradiction.
Hence, we obtain that $\lim_{p\to\infty}\tilde{C}_0(p)=1/2$ .

To complete the proof of (\ref{eq:52}) we need to extend (\ref{eq:56})
to every neighborhood of $r=0$.  Since obtaining an $\OO(\epsilon_p)$ accuracy in
this neighborhood is a difficult task, we allow for an error of larger
magnitude. Thus, requiring that
\begin{equation}
  \label{eq:27}
\frac{2}{p}\bigg[\tilde{C}_0r^{2\big((1-\epsilon_p)^{-1} -1\big)}
-\frac{1}{2}r^2 + \frac{1}{8}r^4-C\epsilon_p^{\frac12}\bigg] 
 \leq g \leq
\frac{2}{p}\bigg[\tilde{C}_0 -\frac{1}{2}r^2 + \frac{1}{8}r^4+C\epsilon_p^{\frac12}\bigg] 
\,.
\end{equation}
It is easy to show that the lower bound in \eqref{eq:27} provides an
estimate which is $O(\epsilon_p^{\frac12})$-accurate  whenever
$r^2>e^{-\epsilon_p^{-\frac{1}{2}}}$. 
  To complete the proof of \eqref{eq:52}, we just need to obtain an
  $O(\epsilon_p^{\frac12})$-accurate estimate for $g$, valid for $r^2\leq e^{-\epsilon_p^{-\frac{1}{2}}}$.

We argue again by bootstrapping.  We may regroup the terms in
\eqref{eq:91} to get
 \begin{equation}
\label{eq:75}
   -\frac{2}{p}|\nabla u_p|^{2-p}f_p(1-f_p^2)=f_p''\Big( 1+\frac{p-2}{|\nabla
     u_p|^2}|f_p'|^2\Big)+\Big(\frac{f_p'}{r}-\frac{f_p}{r^2}\Big)\Big(1+\frac{p-2}{|\nabla
     u_p|^2}\frac{f_pf_p'}{r}\Big)\,.
 \end{equation}
By  Step~1 in the proof of \rprop{prop:at-zero} we have
$\frac{f_p'}{r}-\frac{f_p}{r^2}>0$, and by \rcor{cor:f''<0},
$f_p''<0$. Hence, 
\begin{displaymath}
   f_p^{\prime\prime} \geq -\frac{2}{p}|\nabla u_p|^{2-p}f_p(1-f_p^2) \,.
\end{displaymath}
It follows that  as long as
\begin{displaymath}
  g \geq g(0) \big(1-\epsilon_p^{\frac{1}{2}}\big) \,, 
\end{displaymath}
we must have, by \eqref{eq:48}  that
 \begin{displaymath}
   f_p^{\prime\prime} \geq - \frac{2}{p} \frac{r}{\left[g(0)
       \big(1-\epsilon_p^{\frac{1}{2}}\big)\right]^{(p-2)/p}}  \,.
 \end{displaymath}
Integrating the above yields, in view of \eqref{eq:58},
\begin{displaymath}
  f_p^\prime(r) \geq f_p^\prime(0) - 4r^2\big(1+2\epsilon_p^{\frac{1}{2}}\big) \,.
\end{displaymath}
 Note that we can replace the constant $4$ by any other constant
 greater than $3$.
Consequently,
\begin{displaymath}
  |\nabla u_p| \geq \sqrt{2}|f_p^\prime| \geq \sqrt{2}\big|f_p^\prime(0) - 4r^2\big(1+2\epsilon_p^{\frac{1}{2}}
  \big)\big| \geq \sqrt{2}|f_p^\prime(0)|\bigg|1 - \frac{4\sqrt{2}r^2}{|\nabla u_p(0)|}^2\big(1+2\epsilon_p^{\frac{1}{2}}
  \big)\bigg| \,.
\end{displaymath}
Hence, 
\begin{displaymath}
  g(r) \geq g(0) \big(1-\epsilon_p^{\frac{1}{2}}\big) \Rightarrow g(r) \geq g(0)\bigg[1- \frac{4\sqrt{2}r^2}{|\nabla u_p(0)|}\big(1+2\epsilon_p^{\frac{1}{2}} \big) \bigg]^p
\end{displaymath}

Applying again \eqref{eq:58} we obtain that as long as
\begin{displaymath}
  r^2 < \frac{1}{8p} \epsilon_p^{\frac{1}{2}}\,,
\end{displaymath}
we have
\begin{equation}
\label{eq:3}
   g (r)\geq g(0) \big(1-\epsilon_p^{\frac{1}{2}}\big)\,.
\end{equation}
On the other hand, 
\begin{equation}
\label{eq:46}
  g(r) \leq g(0) \,.
\end{equation}
Since \eqref{eq:3}, \eqref{eq:46},  and  \eqref{eq:27} are
simultaneously satisfied  at
$r^2=2e^{-\epsilon_p^{-\frac{1}{2}}}$, we obtain
\begin{displaymath}
  \begin{cases}
  \frac{2}{p}\tilde{C}_0(1+C\epsilon_p^{\frac{1}{2}}) \geq g(0)(1-\epsilon_p^{\frac{1}{2}})\\
  \frac{2}{p}\tilde{C}_0(1-C\epsilon_p^{\frac{1}{2}}) \leq  g(0)
  \end{cases}\,.
 \end{displaymath}
Consequently,
\begin{displaymath}
  \bigg|\tilde{C}_0 - \frac{p}{2}g(0)\bigg| \leq C\epsilon_p^{\frac{1}{2}}\,.
\end{displaymath}
Furthermore, in view of \eqref{eq:3} and (\ref{eq:27}) we can safely
state that
\begin{equation}
\label{eq:6}
  \frac{2}{p}\bigg[\tilde{C}_0 -\frac{1}{2}r^2 + \frac{1}{8}r^4\tilde{C}_0-C\epsilon_p^{\frac12}\bigg]
\leq g \leq
\frac{2}{p}\bigg[\tilde{C}_0 -\frac{1}{2}r^2 + \frac{1}{8}r^4+C\epsilon_p^{\frac12}\bigg] 
\,,\;\text{ in }[0,a],\,\forall a\in(0,\sqrt{2}).
\end{equation}
Or
\begin{displaymath}
  \|p(g-g_0)\|\leq |\tilde{C}_0(p)-1|+ C\epsilon_p^{1/2} \,.
\end{displaymath}

\end{proof}

\begin{remark}
  From \eqref{eq:52} we can obtain the next two terms in the
  asymptotic expansion of $f_p$ in the large $p$ limit
  \begin{equation}
\label{eq:78}
    f_p = \frac{r}{\sqrt{2}}\left[ 1 - \frac{\ln p}{p} + \frac{\ln g_0(r)}{p} + o \bigg(
    \frac{1}{p} \bigg)\right] \,.
  \end{equation}
The above expansion is valid in $[0,a]$ for every $a<\sqrt{2}$.
\end{remark}
\begin{remark}
  Note that \eqref{eq:6} is valid for all $r>0$. It is only because of
  \eqref{eq:54} that we have to confine the validity of \eqref{eq:52}
  to closed intervals in $[0,\sqrt{2})$ whose edges do not depend on
  $p$. Note further that by \eqref{eq:60} we have that $\epsilon_p\leq2$ for all
  $r\leq\sqrt{2}$. 
\end{remark}
We can now extend the validity of the above estimate to
$[0,\sqrt{2}-\OO(\sqrt{\ln p/p})]$.
\begin{proposition}
\label{prop-extend}
  There exists $C>0$, which is independent of $p$, such that the estimate \eqref{eq:52} holds for
  every $r\in[0,\sqrt{2}-C(\ln p /p)^{1/2}]$. 
\end{proposition}
\begin{proof}
Let
\begin{displaymath}
  g_{0,C}= \tilde{C}_0 -\frac{1}{2}r^2 + \frac{1}{8}r^4\,.
\end{displaymath}
Suppose that $\tilde{C}_0 $ is such that $
g_{0,C}(\sqrt{2}-\Delta_p)=0$. From the previous lemma we have that
$\Delta_p\to0$ as $p\to\infty$. It is easy to show that,
\begin{displaymath}
  g_{0,C}\Big(\sqrt{2}-\frac{2}{3}\Delta_p\Big)\leq -C\Delta_p^2 \,,
\end{displaymath}
for all $C<1/6$ and for sufficiently large $p$.

Let 
\begin{displaymath}
  |\nabla u_p|\Big(\sqrt{2}-\frac{2}{3}\Delta_p\Big) = 1-\delta_p \,.
\end{displaymath}
Since $f_p/r\leq 1/ \sqrt{2}$, we have $f_p^\prime \big(\sqrt{2}-2\Delta_p/3\big)\geq 1-C\delta_p$. From here
it is easy to show that $\epsilon_p (\sqrt{2}-2\Delta_p/3)\leq C\delta_p$. 
By \eqref{eq:75} we have
\begin{equation}
\label{eq:76}
  f_p^{\prime\prime} \leq  -\frac{C}{p^2}|1-\delta_p|^{2-p}(1-\frac{r}{\sqrt{2}}) +
    Cp^{1/2} \quad \forall r\in[\sqrt{2}-2\Delta_p/3,\sqrt{2}-\Delta_p/3] \,, 
\end{equation}
where we have taken into account the fact that $|\nabla u_p|$ is decreasing
and that
\begin{displaymath}
  \frac{1+\frac{p-2}{|\nabla
     u_p|^2}\frac{f_p'f_p}{r}}{1+\frac{p-2}{|\nabla
     u_p|^2}|f_p'|^2} \leq C p^{1/2} \,.
\end{displaymath}
Integrating \eqref{eq:76} over $[\sqrt{2}-2\Delta_p/3,\sqrt{2}-\Delta_p/3]$
yields
\begin{displaymath}
  -\frac{1}{\sqrt{2}} \leq -C\frac{\Delta_p^2}{p^2}|1-\delta_p|^{-p} + Cp^{1/2}\Delta_p \,,
\end{displaymath}
from which we obtain
\begin{displaymath}
  (1-\delta_p)^p \geq C\frac{\Delta_p^2}{p^{5/2}} \,.
\end{displaymath}
Consequently,
\begin{displaymath}
  \delta_p \leq \frac{5}{2} \frac{\ln p}{p} - 2 \frac{\ln \Delta_p}{p} +
  \frac{C}{p} \,.
\end{displaymath}
We conclude from here that
\begin{equation}
\label{eq:77}
  \epsilon_p(\sqrt{2}-2\Delta_p/3) \leq C \delta_p \leq C \frac{\ln p}{p} \,.
\end{equation}
Since $g$ is positive we obtain by \eqref{eq:56} that
\begin{displaymath}
  \Delta_p \leq C \Big[\frac{\ln p}{p}\Big]^{1/2}\,.
\end{displaymath}
Since $\epsilon_p(a)$ is an increasing function of $a$ \eqref{eq:52} must be
valid in $[0,\sqrt{2}-2\Delta_p/3]$.
\end{proof}

\begin{proof}[Proof of Theorem  \ref{thplarge}]
In view of proposition
\ref{prop-extend} there exists
$C>0$ such that (\ref{eq:48}a) and hence \eqref{eq:78} hold for sufficiently large $p$
whenever $r<\sqrt{2}-C(\ln p/p)^{1/2}$.  From the monotonicity of $f_p$ it follows
that
\begin{displaymath}
  f_p(\sqrt{2}-C/(\ln p/p)^{1/2})\leq f_p(r) \leq 1 \,.
\end{displaymath}
\end{proof}

\section{Stability of the radial solution}
\label{sec:4}
In this section we prove our main stability result for
$u_p=f_p(r)e^{i\theta}$,  the degree one radially symmetric
  solution of
  \begin{equation}
    \frac{p}{2}\nabla\cdot(|\nabla u|^{p-2}\nabla u) + u(1-|u|^2)=0 \,.
  \end{equation}
 A simple computation gives the second variation of
  $E_p$ at $u_p$:
\begin{multline}
   \label{eq:100}
 J_2(\phi) = \int_{\R^2} \Bigg\{ \frac{p}{2}|\nabla u_p|^{p-2}\bigg[ |\nabla\phi|^2 +
   (p-2) \frac{|\Re(\nabla u_p\cdot\nabla\bar{\phi})|^2}{|\nabla u_p|^2} \bigg]\\ +
   2|\Re(u_p\bar{\phi})|^2 - (1-|u_p|^2)|\phi|^2 \Bigg\} \,.
 \end{multline}
 Because of \eqref{eq:100} and analogously to \cite{deetal04}, we
 consider perturbations in the ``natural'' Hilbert space ${\cal H}$
 consisting of functions $\phi\in H^1_{\text{loc}}(\R^2,\R^2)$ for which
 \begin{multline*}
   \int_{\R^2} \int_{\R^2} \Bigg\{ \frac{p}{2}|\nabla u_p|^{p-2}\bigg[ |\nabla\phi|^2 +
   (p-2) \frac{|\Re(\nabla u_p\cdot\nabla\bar{\phi})|^2}{|\nabla u_p|^2} \bigg]\\ +
   2|\Re(u_p\bar{\phi})|^2 +(1-|u_p|^2)|\phi|^2 \Bigg\}<\infty \,.
 \end{multline*}
Note that ${\cal H}$ contains all  ``admissible perturbations'' $\phi$,
i.e., any $\phi$ for which
$E_p(u_p+\phi)<\infty$. 
 Note also that in contrast with the case $p=2$, in our case $p>2$, constant
 functions do belong to ${\cal H}$. Thanks to  the invariance of
 the functional $E_p$ with respect to rotations and translations
 (see \cite{ovsi97}) we have
 \begin{equation}
 \label{eq:13}
  J_2(\phi)=0 \text{ for } \phi= \begin{cases} \frac{\partial u_p}{\partial\theta}=if_pe^{i\theta}\,,\\
 \frac{\partial u_p}{\partial x_1}=\frac{1}{2}(f_p'-\frac{f_p}{r})e^{2i\theta}+\frac{1}{2}(f_p'+\frac{f_p}{r})\,,\\
\frac{\partial u_p}{\partial
  x_2}=-\frac{i}{2}(f_p'-\frac{f_p}{r})e^{2i\theta}+\frac{i}{2}(f_p'+\frac{f_p}{r})\,.
\end{cases}
 \end{equation}
 Indeed, this leads to
 the equality cases in the next theorem.
\begin{theorem}
\label{th:stable}
 For every  $2<p\leq4$ the radially symmetric solution $u_p$ is stable in
 the sense that $J_2(\phi)\geq 0$ for all $\phi\in {\cal H}$. Moreover, we have
 $J_2(\phi)=0$ if and only if 
 \begin{equation}
   \label{eq:96}
\phi=c_0\frac{\partial u_p}{\partial\theta}+c_1\frac{\partial
   u_p}{\partial x_1}+c_2\frac{\partial u_p}{\partial x_2},\text{ for some constants $c_0,c_1,c_2\in\R$.}
 \end{equation}
\end{theorem}

Following \cite{mi95} we represent each $\phi$ by its Fourier expansion
\begin{equation}
\label{eq:95}
  \phi=\sum_{n=-\infty}^\infty \phi_n(r)e^{in\theta} \,.
\end{equation}
Substituting into \eqref{eq:100} we obtain
\begin{equation}
\label{eq:22}
  \frac{1}{2\pi} J_2(\phi) = E_1(\phi_1) + \sum_{n=2}^\infty  E_n(\phi_n,\phi_{2-n}) \,,
\end{equation}
in which
\begin{subequations}
  \begin{multline}
\label{eq:79}
    E_1(\phi_1) = \int_0^\infty \Bigg\{\frac{p}{2} |\nabla u_p|^{p-2} \bigg[ |\phi_1^\prime|^2 +
    \frac{1}{r^2}|\phi_1|^2 + (p-2) \frac{\big|\Re\big(f_p^\prime\phi_1^\prime
      +\frac{f_p\phi_1}{r^2}\big)\big|^2} {|\nabla u_p|^2} \bigg] \\
+  2f_p^2 |\Re \phi_1|^2 - (1-f_p^2)|\phi_1|^2 \Bigg\} rdr \,,
\end{multline}
and
\begin{multline}
\label{eq:80}
    E_n(\phi_n,\phi_{2-n}) = \int_0^\infty \Bigg\{\frac{p}{2} |\nabla u_p|^{p-2} \bigg[
    |\phi_n^\prime|^2 + |\phi_{2-n}^\prime|^2+
    \frac{n^2}{r^2}|\phi_n|^2+ \frac{(2-n)^2}{r^2}|\phi_{2-n}|^2
    + \\ \frac{1}{2}(p-2) \frac{\big|f_p^\prime(\bar{\phi}_n^\prime+\phi_{2-n}^\prime)
      +\frac{f_p}{r^2}(n\bar{\phi}_n+(2-n)\phi_{2-n})\big|^2} {|\nabla u_p|^2} \bigg] \\
+  f_p^2 | (\bar{\phi}_n+ \phi_{2-n})|^2 - (1-f_p^2)(|\phi_n|^2 + |\phi_{2-n}|^2) \Bigg\} rdr \,.
\end{multline}
\end{subequations}

A necessary and sufficient condition for the positive definiteness of
$J_2$ is that the $E_n$'s are all positive definite. An appropriate
Hilbert space for the study of the functionals $\{E_n\}$ is 
\begin{displaymath}
  {\cal S} = \{ \phi\in H^1_{loc}(\R_+,\C)\cap L^2_r(\R_+,\C)\,:\, \int_0^\infty \frac{p}{2} |\nabla u_p|^{p-2} \Big[ |\phi^\prime|^2 +
    \frac{1}{r^2}|\phi|^2\Big]  \,rdr<\infty \} \,.
\end{displaymath}
We also denote by $\tilde{\cal S}$ the space of real-valued functions
  in ${\cal S}$.
\subsection{$n\neq2$}
\label{sec:4.1}
We consider first the case $n=1$.
\begin{lemma}
\label{lem:E1}
  \begin{equation}
    \label{eq:101}
\inf_{\phi\in{\cal S} } E_1(\phi)= 0 \,.
  \end{equation}
Furthermore, the minimum in \eqref{eq:101} is attained only for
$\phi=cif_p$, for any real constant  $c$.
\end{lemma}
\begin{proof}
Since $E_1(i|\phi|) \leq E_1(\phi)$ for every $\phi$ for which $E_1(\phi)<\infty$, with
strict inequality unless $\phi$ takes only purely imaginary values, we may
consider instead of $E_1$ the following functional
\begin{displaymath}
  \tilde{E}_1(\phi) =  \int_0^\infty \bigg\{\frac{p}{2} |\nabla u_p|^{p-2} \Big[ |\phi^\prime|^2 +
    \frac{1}{r^2}|\phi|^2\Big] - (1-f_p^2)|\phi|^2 \bigg\} rdr \,,
\end{displaymath}
over $\tilde{\cal S}$.
  Consider first $\phi\in C_c^\infty(0,\infty)$  and set $  \phi = f_pw $.
Integration by parts, with the aid of \eqref{eq:30} yields
  \begin{equation}
  \label{eq:102}
F_1(w) = \tilde{E}_1(f_pw) = \int_0^\infty \frac{p}{2} |\nabla u_p|^{p-2} f_p^2|w^\prime|^2 rdr \,.
\end{equation}
 A standard use of cut-off functions yields that \eqref{eq:102} holds
 also for smooth $\phi=f_pw$ with compact support in $[0,\infty)$ (i.e, the
 support  may
 contain the origin). Finally, by density of smooth maps with compact
 support in $[0,\infty)$  in $\tilde{\cal S}$ it follows that
 \eqref{eq:102} continues to hold for $\phi=f_pw\in\tilde{\cal S}$. Therefore, 
  $\tilde{E}_1(\phi)\geq 0$ for all $\phi\in \tilde{\cal S}$ and 
   $F_1(w)=0$ if and only if $w\equiv\text{const}$.
\end{proof}

\noindent We now consider the case $n\geq3$. 
\begin{proposition}
\label{prop:E3}
 For each $n\geq 3$ we have
  \begin{displaymath}
   E_n(u_1,u_2)> 0 \text{ for all } (u_1,u_2)\in \tilde{\cal
      S}\times\tilde{\cal S}\setminus \{(0,0)\}\,.
\end{displaymath}
\end{proposition}
\begin{proof}
  The result follows right away from the previous lemma  and the inequality
  \begin{displaymath}
    E_n(u_1,u_2) \geq \tilde{E}_1( |u_1| ) + \tilde{E}_1( |u_2| )\,,
  \end{displaymath}
 with strict inequality, unless $u_j\equiv 0,\,j=1,2$.
\end{proof}

\subsection{$n=2$}
\label{sec:4.2}
It is easy to reduce the analysis of $E_2$ to that of a functional
acting on real-valued functions.  Indeed,  writing  a complex-valued function $\phi$ as
$\phi=\phi^R+i\phi^I$, we have
\begin{displaymath}
  E_2(\phi_2,\phi_0) = E_2^R(\phi_2^R,\phi_0^R) + E_2^I(\phi_2^I,\phi_0^I) \,, 
\end{displaymath}
where
\begin{displaymath}
  E_2^R(\phi_2^R,\phi_0^R) = E_2(\phi_2^R,\phi_0^R), \quad E_2^I(\phi_2^I,\phi_0^I) =
  E_2(i\phi_2^I,i\phi_0^I) \,.
\end{displaymath}
Clearly,
\begin{displaymath}
  E_2(i\phi_2^I,i\phi_0^I)=E_2^R(-\phi_2^I,\phi_0^I) \,.
\end{displaymath}
Hence, 
\begin{equation}
  \label{eq:85}
 E_2(\phi_2,\phi_0)=E_2^I(-\phi_2^R,\phi_0^R) +E_2^I(\phi_2^I,\phi_0^I)\,,
\end{equation}
 and it suffices to study the minimization to  the functional $E_2^I$ over $\tilde{\cal S}\times \tilde{\cal S}$.

From \eqref{eq:13} and \eqref{eq:22} it follows that  the functions
\begin{equation}
\label{eq:86}
    \Phi_0 = f_p^\prime + \frac{f_p}{r}\mbox{ and }\Phi_2 = - f_p^\prime + \frac{f_p}{r}\,,
\end{equation}
 satisfy
\begin{displaymath}
E_2^R(-\Phi_2,\Phi_0)=E_2^I(\Phi_2,\Phi_0)=0\,.
\end{displaymath}
 We next claim:
 \begin{proposition}
 \label{prop:E2}
   For $p\in(2,4]$ we have $E_2^I(\phi_2,\phi_0)\geq 0$ for every $(\phi_2,\phi_0)\in\tilde{\cal S}\times
   \tilde{\cal S}$ with equality if and only if $(\phi_2,\phi_0)=c(\Phi_2,\Phi_0)$
   for some $c\in\R$ (see \eqref{eq:86}).
 \end{proposition}
For the proof of \rprop{prop:E2} we shall need some preliminary
results. First,  by \eqref{eq:80} we have
\begin{multline}
    E_2^I(\phi_2,\phi_0) = \int_0^\infty \Bigg\{\frac{p}{2} |\nabla u_p|^{p-2} \bigg[
    (\phi_2^\prime)^2 + (\phi_0^\prime)^2 +  \frac{4}{r^2}(\phi_2)^2
    \\ + \frac{1}{2}(p-2) \frac{\big(f_p^\prime(\phi_0^\prime-\phi_2^\prime)
      - 2 \frac{f_p}{r^2}\phi_2)\big)^2} {|\nabla u_p|^2} \bigg] \\
+  f_p^2 (\phi_0- \phi_2)^2 - (1-f_p^2)\big((\phi_2)^2 + (\phi_0)^2\big) \Bigg\} rdr \,.
\end{multline}
It is more convenient to consider an alternative form by applying the transformation
\begin{displaymath}
  A= \phi_0 + \phi_2, \quad B=\phi_0-\phi_2,
\end{displaymath}
to obtain
\begin{multline}
  \label{eq:103}
  E_2^I(\phi_0,\phi_2) = F_2(A,B) := \int_0^\infty \Bigg\{\frac{p}{4} |\nabla u_p|^{p-2} \bigg[
    (A^\prime)^2 + (B^\prime)^2 +  \frac{2}{r^2}(A-B)^2
    \\ ~+ (p-2) \frac{\big(f_p^\prime B^\prime 
      - \frac{f_p}{r^2}(A-B)\big)^2} {|\nabla u_p|^2} \bigg] \\
+  f_p^2 B^2 - \frac{1}{2}(1-f_p^2)(A^2 + B^2) \Bigg\} rdr \,.
\end{multline}
Clearly,
\begin{equation}
\label{eq:81}
  F_2(f_p/r,f_p^\prime) =0\,.
\end{equation}
 The ``problematic term'' in \eqref{eq:103} is the one involving the
 mixed product $AB'$. The difficulty in handling this term is the
 obstacle for determining the positivity of $F_2$ for every $p>2$. We
 were able to overcome this  difficulty only in the case $p\in(2,4]$ thanks to the
 following lemma. 
\begin{lemma}
\label{lem:G2}
 We have 
 \begin{equation}
\label{eq:84}
F_2(A,B)=G_2(A,B)+\int_0^\infty \frac{p(p-2)}{4} |\nabla u_p|^{p-2}
    \frac{\big(hA^\prime -\frac{1}{r}(h^2A-B)\big)^2} {1+h^2} rdr   \,,
\end{equation}
 with
\begin{multline}
\label{eq:98}
G_2(A,B) = \int_0^\infty \Bigg\{\frac{p}{4} |\nabla u_p|^{p-2} \bigg[
    (A^\prime)^2 + (B^\prime)^2 +  \frac{2}{r^2}(A-B)^2
    \\ + (p-2)
    \frac{h^2((B^\prime)^2-(A^\prime)^2) +
      \frac{1-h^4}{r^2}A^2 -\frac{2}{r^2}(1-h^2)AB}  {1+h^2} \Bigg] \\
    + \frac{p(p-2)}{4r}[H^\prime(2AB- B^2) -(h^2H)^\prime A^2]  \\
+  f_p^2 B^2 - \frac{1}{2}(1-f_p^2)(A^2 +
  B^2) \Bigg\} rdr\,,
  \end{multline}
where
\begin{displaymath}
  H= \frac{h}{1+h^2}|\nabla u_p|^{p-2} ~\text{ and }~ h=rf_p^\prime/f_p\text{ (as in \eqref{eq:1000})}\,.
\end{displaymath}
Moreover, $G_2(f_p/r,f_p')=0$ and the pair $(f_p/r,f_p')$ solves the
Euler-Lagrange equations associated with $G_2$.
\end{lemma}
\begin{proof}
   First, a direct computation gives the identity
   \begin{multline}
\label{eq:82}
  \frac{(f_p^\prime B^\prime 
      - \frac{f_p}{r^2}(A-B))^2} {|\nabla u_p|^2} \\=
 \frac{h^2(|B^\prime|^2-|A^\prime|^2) -\frac{2h}{r}[(AB)^\prime - BB^\prime -h^2AA^\prime]
     + \frac{1-h^4}{r^2}|A|^2 -\frac{2}{r^2}(1-h^2)AB}  {1+h^2}\\
+\frac{(hA^\prime -\frac{1}{r}(h^2A-B))^2} {1+h^2}\,.
   \end{multline}
Next, integration by parts yields
\begin{multline}
\label{eq:83}
  \int_0^\infty  |\nabla u_p|^{p-2} \Big\{ \frac{ -\frac{2h}{r}[(AB)^\prime - BB^\prime -h^2AA^\prime]}
   {1+h^2}\Big\}r\,dr\\=
 \int_0^\infty \Big\{H^\prime(2AB- B^2) -(h^2H)^\prime A^2 \Big\}\,dr
\end{multline}
Using \eqref{eq:82}--\eqref{eq:83} in conjunction with \eqref{eq:103}
leads to \eqref{eq:84}--\eqref{eq:98}. Finally, a direct computation
shows that the integrand in the integral on the right-hand-side of
\eqref{eq:84} is identically zero for $A=f_p/r$ and $B=f_p'$, and the
last assertion of the lemma follows.
\end{proof}
\begin{proof}[Proof of \rprop{prop:E2}] In view of \rlemma{lem:G2} it
  suffices to show that
  \begin{multline}
    \label{eq:87}
 G_2(u,v)\geq 0,\,\forall(u,v)\in\tilde{\cal S}\times\tilde{\cal
   S},\text{ with equality iff: }\\u=\phi:=f_p/r\text{ and }v=\psi:=f_p'\,.
  \end{multline}
 We write $G_2$ in the form
  \begin{equation*}
 G(u,v)=\int_0^\infty \big(\alpha(r)u'^2 +\beta(r)v'^2+a(r)u^2+2b(r)uv+c(r)v^2\big)\,dr\,.
\end{equation*}
 The properties of the coefficients which are important to us are
 \begin{equation}
   \label{eq:89}
\alpha(r),\beta(r)>0\text{ and }b(r)<0,~\text{ for }r>0\,.
 \end{equation}
Indeed, clearly $\beta(r)>0$. Next,
\begin{equation*}
  \alpha(r)=\frac{p}{4} |\nabla u_p|^{p-2} r\Big(1-(p-2)\frac{h^2}{1+h^2}\Big)>0\,,
\end{equation*}
 provided $p\leq 4$, since $0<h<1$ by Step~1 of \rprop{prop:at-zero} and \rprop{le2a}.
Finally, 
\begin{equation*}
   b=r\Big\{ \frac{p}{4}|\nabla u_0|^{p-2}\big[-\frac{2}{r^2}-\frac{(p-2)}{r^2}(1-h^2)\big]+\frac{p(p-2)}{4r}H'\Big\}<0\,,
 \end{equation*}
since $0<h(r)<1$,  and 
\begin{equation*}
  H'=|\nabla
  u_p|^{p-2}\frac{(1-h^2)h'}{(1+h^2)^2}+(p-2)\frac{h}{1+h^2}\Big(\big(\frac{f_p}{r}\big)
\big(\frac{f_p}{r}\big)'+f_p'f_p''\Big)<0\,,
\end{equation*}
since $h'\leq 0$ by \rlemma{lem:monotonicity2} and both $f_p^\prime$ and
$f_p/r$ are decreasing (as we noted already before, by \rlemma{lem:monotonicity2} and Step~2
of the proof of  \rprop{prop:at-zero}).

 By \rlemma{lem:G2} we know that $\phi$ and $\psi$ satisfy
 \begin{equation}
\label{eq:88}
  \left\{ 
 \begin{aligned}
   -(\alpha\phi')'+a\phi+b\psi&=0,\\
   -(\beta\psi')'+c\psi+b\phi&=0.
\end{aligned}
\right.
 \end{equation}
\end{proof}
 We consider first $u,v\in C^\infty_c(0,\infty)$. By Picone's identity
\begin{align}
  (u')^2-(\frac{u^2}{\phi})'\phi'&=(u'-(u/\phi)\phi')^2\geq 0 \label{eq:90}\\
 (v')^2-(\frac{v^2}{\psi})'\psi'&=(v'-(v/\psi)\psi')^2\geq 0\,.\label{eq:92}
\end{align}
Multiplying  \eqref{eq:90}--\eqref{eq:92} by $\alpha$ and $\beta$ respectively,  
 applying integration by parts and using \eqref{eq:88} we obtain
\begin{equation}
\label{eq:104}
\begin{aligned}
  0&\leq \int_0^\infty \alpha(u')^2-\alpha(\frac{u^2}{\phi})'\phi'+\beta(v')^2-\beta(\frac{v^2}{\psi})'\\
&=\int_0^\infty
  \alpha(u')^2+\frac{u^2}{\phi}(a\phi+b\psi)+\beta(v')^2+\frac{v^2}{\psi}(c\psi+b\phi)\\
&=\int_0^\infty \alpha u'^2 +\beta v'^2+au^2+cv^2+b(u^2\frac{\psi}{\phi}+v^2\frac{\phi}{\psi})\\
&=G(u,v)+\int_0^\infty b\Big(u\big(\frac{\psi}{\phi}\big)^{1/2}-v\big(\frac{\phi}{\psi}\big)^{1/2}\Big)^2\,.
\end{aligned}
\end{equation}
From \eqref{eq:104} and a density argument we conclude that
\begin{equation*}
 G(u,v)\geq \int_0^\infty
 (-b)\Big(u\big(\frac{\psi}{\phi}\big)^{1/2}-v\big(\frac{\phi}{\psi}\big)^{1/2}\Big)^2\,,~\forall u,v\in\tilde{\cal S}\,,
\end{equation*}
and \eqref{eq:87} follows.

Next we are ready to present the proof of our main stability theorem.
\begin{proof}[Proof of \rth{th:stable}] Representing each $\phi\in{\cal H}$
  by its Fourier expansion \eqref{eq:95}, we have by \eqref{eq:22},
  \rlemma{lem:E1}, \rprop{prop:E3}, \eqref{eq:85} and \rprop{prop:E2}
  that  $J_2(\phi)\geq 0$. Furthermore, by the equality cases in
  \rlemma{lem:E1}, \rprop{prop:E3} and \rprop{prop:E2} we have
  $J_2(\phi)=0$ iff $\phi=\phi_0+\phi_1e^{i\theta}+\phi_2e^{2i\theta}$  where 
  \begin{equation*}
    \phi_1=a_1if_p\,,\quad (\phi_2^I,\phi_0^I)=a_2(\Phi_2,\Phi_0)\,\text{ and }
    (-\phi_2^R,\phi_0^R)=a_3(\Phi_2,\Phi_0)\,,\quad\text{ with }a_1,a_2,a_3\in\R\,.
  \end{equation*}
  It is easy to verify that these relations
  are equivalent to \eqref{eq:96}.
  \end{proof}
\bibliography{pGL}

\end{document}